\documentclass[11pt]{article}
\usepackage[T1]{fontenc}
\usepackage{lmodern}
\usepackage[utf8]{inputenc}
\usepackage[english]{babel}
\usepackage[a4paper,margin=30mm]{geometry}
\usepackage{amsmath,amssymb,amsthm,mathtools}
\usepackage{microtype,enumitem,booktabs,xcolor,array,tabularx}
\usepackage[colorlinks=true,linkcolor=blue!50!black,citecolor=blue!50!black,urlcolor=blue!50!black]{hyperref}
\hypersetup{pdftitle={Small-speed solitary-wave stability and global well-posedness for a Bessel-regularized Airy equation},pdfauthor={Germán Preciado López}}
\newtheorem{theorem}{Theorem}[section]
\newtheorem{proposition}[theorem]{Proposition}
\newtheorem{lemma}[theorem]{Lemma}

\theoremstyle{definition}

\numberwithin{equation}{section}
\newcommand{\R}{\mathbb R}
\newcommand{\la}[1]{\langle #1\rangle}
\newcommand{\norm}[1]{\left\|#1\right\|}

\title{\bfseries Small-speed solitary-wave stability\\and global well-posedness for a\\Bessel-regularized Airy equation}
\author{Germán Preciado López\\\small Department of Mathematics, Universidad Nacional de Colombia\\\small Bogotá, Colombia}
\date{October 2026}
\begin{document}
\maketitle
\begin{abstract}
We prove orbital stability of minimizing solitary waves at sufficiently
small positive speeds for an Airy equation with a quadratic interaction
filtered by a Bessel potential of order one. The stability hypotheses
are verified: constrained minimization gives one negative direction,
the small-speed KdV limit and operator-norm perturbation of the profile
resolvent give a one-dimensional translation kernel, and a differentiable
rescaled branch with uniformly bounded derivative gives a positive slope
of the quadratic invariant. A conserved Lyapunov functional then proves
orbital stability in the three-halves-derivative energy space and the
perturbation also gives uniqueness of centered minimizers at small speed.
The dynamical framework is supplied by a direct Bourgain-space bilinear
estimate, local well-posedness in the square-integrable class,
unconditional uniqueness from one-quarter derivative onward, and global
well-posedness in the half-derivative energy space. For every prescribed
positive speed, constrained minimization constructs smooth positive
profiles that are symmetric, strictly monotone away from their center,
and exponentially localized. At general positive speeds the stability
criterion remains conditional on nondegeneracy and the slope condition.
\end{abstract}
\medskip
\noindent\textbf{Keywords:} Bessel-regularized Airy equation; Bourgain spaces;
global well-posedness; solitary waves; orbital stability;
nondegeneracy; small-speed KdV limit.

\smallskip
\noindent\textbf{2020 Mathematics Subject Classification:}
Primary 35Q53; Secondary 35R11, 35C07, 35B35.

\section{Introduction and main results}
We consider the real-valued equation
\begin{equation}\label{eq:model}
\boxed{u_t+u_{xxx}+B(u^2)=0,\qquad x\in\R,\qquad
B=\partial_x(1-\partial_x^2)^{-1/2}}
\end{equation}
Set $J=(1-\partial_x^2)^{1/2}$. The symbol of $B$ is
\[
m(\xi)=i\xi\la{\xi}^{-1},\qquad \la{\xi}=(1+\xi^2)^{1/2}.
\]
The central question is whether the spectral and slope conditions for
orbital stability can be verified for the nonlocal profile operator
$A_c=J(c-\partial_x^2)$. We answer it for all sufficiently small positive
speeds. The proof combines the KdV limit of every centered minimizing
profile with a perturbation of the smoothing resolvent, and controls
the derivative of the rescaled branch before concluding positivity of
the slope. The evolution and variational results below supply the global
flow and the profiles required for this stability argument.

The operator is bounded on every $H^s$, but its quadratic interaction is not bounded from $L^2$ to $L^2$. Indeed, for a nonzero Schwartz function $\phi$ and $u_N(x)=N^{1/2}\phi(Nx)$,
\[
\norm{u_N}_2=\norm\phi_2,\qquad
\norm{B(u_N^2)}_2^2
=N\int\frac{N^2\eta^2}{1+N^2\eta^2}|\widehat{\phi^2}(\eta)|^2\,d\eta
\sim N\norm{\phi^2}_2^2.
\]
Thus an instantaneous spatial estimate does not close an $L^2$ fixed point. The Airy phase and the zero of $m$ at the origin provide the missing space--time bound.

\subsection{Position relative to earlier work}
The restriction-space framework originates in Bourgain's dispersive theory
\cite{bourgain}. The KdV bilinear estimate of Kenig, Ponce and Vega
\cite{kpv96} and their contraction approach to generalized KdV
\cite{kpv93} are central precedents. We prove the estimate for the
order-zero force $B(u^2)$ directly: the derivative KdV estimate is not
being asserted for this nonlinearity. Smooth approximation and persistence
have classical antecedents in Bona and Smith \cite{bonasmith}; general
dispersive and semigroup backgrounds are given in \cite{tao,pazy}.

Related two-dimensional Benjamin--Ono equations were studied by Preciado
López and Soriano Méndez \cite{preciadosorianoC,preciadosorianoS}.
Those papers concern different dispersive symbols and dimensions;
the author's filtered-slope equation \cite{preciado2026} has a different
nonlinear force and coercive invariant.

For solitary waves, concentration--compactness goes back to Lions
\cite{lions}; nonlocal variational constructions and their relation to
stability are developed in \cite{albert,arnesen,ehrnstrom}. Recent work
on Bessel-type operators and nonlocal cubic interactions appears in
Marstrander \cite{marstrander}.
The splitting proof below is elementary and is related to the functional
splitting principle of Brézis and Lieb \cite{brezislieb}.

A closer antecedent is Marstrander's theory of cyclically symmetric
Coifman--Meyer nonlinearities \cite{marstranderCM}. Indeed, the change
$\psi=J^{1/2}\phi$ transforms our profile equation into
\[
 (c-\partial_x^2)\psi=N(\psi,\psi),\qquad
 N(f,g)=J^{-1/2}\big[(J^{-1/2}f)(J^{-1/2}g)\big].
\]
Up to the Fourier convolution constant, its bilinear symbol is
\[
 n(\xi-\eta,\eta)
 =\langle\xi\rangle^{-1/2}
  \langle\xi-\eta\rangle^{-1/2}\langle\eta\rangle^{-1/2}.
\]
This symbol is separable and cyclically symmetric in
$-\xi,\xi-\eta,\eta$. With linear symbol $m(\xi)=\xi^2$, the orders
$s=s'=2$ and $r_1=r_2=r_3=-1/2$ satisfy Assumptions
$\mathfrak M$ and $\mathfrak N$ of \cite{marstranderCM}; constant
normalization factors can be absorbed into the three factors of the
cyclic sum. That theorem constructs smooth waves at every prescribed
mass $\frac12\|\psi\|_2^2=\mathcal P(\phi)$, with speed arising as a
Lagrange multiplier ($\nu=-c$ in its profile convention).
Here we prescribe $c>0$ directly and establish symmetry, localization,
and small-speed uniqueness of centered minimizers. The principal
stability contribution is the verification of nondegeneracy and the
positive slope, together with a global flow on which the Lyapunov
argument applies. We therefore place existence within the established
nonlocal variational theory and concentrate on these additional steps.

The symmetry argument uses the integral moving-plane method, as in
\cite{chenliou}, but is proved here for the particular integrable
resolvent kernel. The general stability criterion follows the
energy--momentum strategy of Grillakis, Shatah and Strauss \cite{gss}
through a direct Lyapunov proof. We establish the negative index, but
do not infer nondegeneracy or the slope of $\mathcal P$ from symmetry: both are verified at small speed by perturbing the KdV resolvent.

\subsection{What the Bessel filter changes}
The equation has the Airy dispersion of KdV, with a regularized quadratic
transport flux. Its multiplier has two distinct regimes:
\[
 \frac{i\xi}{\sqrt{1+\xi^2}}
 =i\xi-\frac{i\xi^3}{2}+O(|\xi|^5)\quad(\xi\to0),
 \qquad |m(\xi)|\longrightarrow1\quad(|\xi|\to\infty).
\]
Thus long waves see the KdV quadratic transport to leading order,
whereas short-wave nonlinear differentiation is bounded in frequency.
The multiplier is imaginary and odd: this is a conservative spatial
filter, not viscous damping, and the linear Airy dispersion is unchanged.
The unit length in $J$ sets a crossover scale. Replacing it by a
length $\ell>0$ gives $(1-\ell^2\partial_x^2)^{1/2}$ and makes that
scale explicit. No derivation from a particular natural medium is claimed.

The choice also supplies compatible analytical structures: multiplication
by $J$ cancels the filter in the cubic energy calculation; the profile
resolvent factors as $(c-\partial_x^2)^{-1}J^{-1}$; and each factor
has a positive kernel. These facts yield the global $H^{1/2}$ invariant,
symmetry, and exponential profile localization. A general bounded odd
multiplier need not retain this combination. Bessel regularization is
not the only possible conservative filter, but its symbol makes these
properties verifiable together.

The low-frequency expansion also motivates the small-speed profile
rescaling $\phi_c(x)=c q_c(\sqrt c\,x)$. Theorem~\ref{thm:kdvlimit}
justifies its KdV limit for all centered minimizers. This limit is the
starting point of the resolvent perturbation used to prove stability.

\subsection{Comparison with KdV and Benjamin--Ono}
We normalize the comparison equations as
\[
 \text{KdV: }u_t+u_{xxx}+\partial_x(u^2)=0,\qquad
 \text{BO: }u_t-\mathcal H u_{xx}+\partial_x(u^2)=0,
\]
where $\widehat{\mathcal Hh}=-i\operatorname{sgn}(\xi)\widehat h$.
Their positive-speed profile equations are
$(-\partial_x^2+c)q=q^2$ and $(|D|+c)q=q^2$, respectively.
For this normalization the familiar profiles are
\[
 q_c^{\rm KdV}(x)=\frac{3c}{2}\operatorname{sech}^2
       \left(\frac{\sqrt c\,x}{2}\right),\qquad
 q_c^{\rm BO}(x)=\frac{2c}{1+c^2x^2}.
\]
The first formula follows by direct differentiation. The second follows
from $|D|(1+x^2)^{-1}=(1-x^2)(1+x^2)^{-2}$ and scaling.
Uniqueness for the BO profile is a separate theorem of Amick and Toland
\cite{amicktolland}; it does not transfer to $J(c-\partial_x^2)$.

\begin{center}
\setlength{\tabcolsep}{5pt}
\renewcommand{\arraystretch}{1.3}
\begin{tabularx}{\textwidth}{|>{\raggedright\arraybackslash}p{.21\textwidth}|
>{\raggedright\arraybackslash}X|
>{\raggedright\arraybackslash}X|
>{\raggedright\arraybackslash}X|}
\hline
\textbf{Property} & \textbf{KdV} & \textbf{Benjamin--Ono} & \textbf{Present equation}\\\hline
Profile operator & $c-\partial_x^2$ & $c+|D|$ & $J(c-\partial_x^2)$\\\hline
Conserved quadratic quantity & $\frac12\|u\|_2^2$ & $\frac12\|u\|_2^2$ & $\frac12\|u\|_{H^{1/2}}^2$\\\hline
Action form domain & $H^1$ & $H^{1/2}$ & $H^{3/2}$\\\hline
Positive profile tail & Exponential & Algebraic & Exponential\\\hline
Profile uniqueness & ODE classification\textsuperscript{a} & Amick--Toland \cite{amicktolland} & Centered minimizers at small speed\textsuperscript{b}\\\hline
Orbital stability & Classical precedent & Separate spectral theory & Proved at small speed\textsuperscript{b}\\\hline
\end{tabularx}
\end{center}
\noindent\textsuperscript{a}The KdV ODE classification is proved in
Theorem~\ref{thm:kdvlimit}.
\textsuperscript{b}Theorem~\ref{thm:smallspeedstability}; no assertion
at arbitrary positive speeds is made here.

The table concerns these normalizations and the indicated profile
families, not all possible solutions. Low-regularity KdV theory extends
below $L^2$ \cite{kpv96}, and BO has a global $L^2$ theory
\cite{ionescukenig}. Our $L^2$ local result is therefore not a new
threshold for either classical equation; nor does the conserved
$H^{1/2}$ norm provide global control of arbitrary $L^2$ data.
We extend variational, symmetry and stability arguments to
a different profile operator, not the integrable structure or complete
classification of the classical models.

\subsection{Statements}
The Fourier transform is unitary with phase $e^{-ix\xi}$, and
$\norm u_{H^s}^2=\int\la\xi^{2s}|\widehat u|^2$. Let $S(t)$ be the unitary Airy group, $\widehat{S(t)h}=e^{it\xi^3}\widehat h$. For a fixed $1/2<b<1$, set
\begin{equation}\label{eq:X}
\norm u_{X^{s,b}}=\norm{\la\xi^s\la{\tau-\xi^3}^b\widetilde u(\xi,\tau)}_2.
\end{equation}
Restriction norms are defined in Section~\ref{sec:cauchy}.

\begin{theorem}\label{thm:flow}
For each $s\ge0$ and $u_0\in H^s(\R)$, the equation has a local mild solution in $C([-T,T];H^s)\cap X^{s,b}([-T,T])$. Solutions are unique in the restriction class $X^{0,b}$; for $s\ge1/4$, uniqueness holds among all mild solutions in $C([-T,T];H^s)$, without a restriction-space assumption. The time is uniform on bounded sets of data, the flow is locally Lipschitz on a common interval, and higher nonnegative Sobolev regularity persists on the interval of the lower-regularity solution.

For real data $u_0\in H^s$, $s\ge1/2$, the solution is global and
\begin{equation}\label{eq:Pconservation}
\mathcal P(u(t))=\mathcal P(u_0),\qquad
\mathcal P(u)=\frac12\la{Ju,u}=\frac12\norm u_{H^{1/2}}^2.
\end{equation}
In particular, the global $H^{1/2}$ solution is unconditionally unique
among $C_tH^{1/2}$ mild solutions. No global assertion for arbitrary
$L^2$ data is made.
\end{theorem}
We call $\mathcal P$ the quadratic invariant, to distinguish it from the
usual $L^2$ quadratic quantity of KdV. It also has the translation-generator
role associated with this equation's Hamiltonian operator: for smooth
$u$, $\mathcal P'(u)=Ju$ and $B\mathcal P'(u)=\partial_xu$.
The term energy--momentum method below refers to this Hamiltonian role.

\begin{theorem}\label{thm:waves}
For every $c>0$ there is a nonzero real profile $\phi_c\in\bigcap_{r\ge0}H^r(\R)\cap L^1(\R)$ with $\phi_c>0$, $\phi_c(x)\to0$ as $|x|\to\infty$, and
\begin{equation}\label{eq:profile}
A_c\phi_c=\phi_c^2,\qquad A_c=J(c-\partial_x^2),\qquad
c\int\phi_c=\int\phi_c^2.
\end{equation}
Then $u(t,x)=\phi_c(x-ct)$ solves the equation. The minimizing sequences used to construct these profiles are compact in $H^{3/2}$ modulo translations. Every such positive profile is even about a point and strictly decreasing away from that point. For every $0<\beta<\min\{\sqrt c,1\}$ there is $C_{c,\beta}$ such that $\phi_c(x)\le C_{c,\beta}e^{-\beta|x|}$ after centering. For sufficiently small $c>0$, the centered minimizing profile is unique and its translation orbit is orbitally stable in $H^{3/2}$, by Theorem~\ref{thm:smallspeedstability}. For general positive speeds the criterion in Theorem~\ref{thm:conditionalstability} remains conditional; uniqueness of all positive profiles is not asserted.
\end{theorem}

The main stability result is Theorem~\ref{thm:smallspeedstability}.
For every sufficiently small $c>0$, it verifies the conditions
\[
 n_-(L_c)=1,\qquad
 \ker L_c=\operatorname{span}\{\phi_c'\},\qquad
 \frac{d}{dc}\mathcal P(\phi_c)>0,
\]
where $n_-$ denotes the number of negative eigenvalues counted with
multiplicity. Constrained minimization supplies the first condition.
The KdV limit and resolvent perturbation supply the second; uniform
derivative control of the rescaled branch supplies the third. These
verified conditions give coercivity modulo translation for a conserved
penalized action, and hence orbital stability in $H^{3/2}$.
Theorem~\ref{thm:flow} provides the global, unconditionally unique flow
in $H^{1/2}$ and its higher-regularity persistence; Theorem~\ref{thm:waves}
provides the minimizing profiles. Section~\ref{sec:estimates} establishes
the space--time estimates, Section~\ref{sec:cauchy} proves the evolution
theory and conservation laws, Section~\ref{sec:waves} develops the
profiles and their KdV limit, and Section~\ref{sec:stability} proves
orbital stability and verifies its hypotheses at small speed.

\section{Space--time estimates}\label{sec:estimates}
For $b>1/2$, the interaction representation gives
\begin{equation}\label{eq:interaction}
\norm u_{X^{s,b}}=\norm{S(-t)u(t)}_{H^b(\R_t;H^s_x)},\qquad
\sup_t\norm{u(t)}_{H^s}\le C_b\norm u_{X^{s,b}}.
\end{equation}
Indeed, conjugation shifts the temporal Fourier variable by $\xi^3$. Fourier inversion and Cauchy--Schwarz in time use $\int\la\tau^{-2b}\,d\tau<\infty$. Dominated convergence in the inversion formula proves continuity. Density extends the result beyond Schwartz functions.

\begin{lemma}\label{lem:integrals}
For $p>1$,
\[
\int\la r^{-p}\la{A-r}^{-p}\,dr\le C_p\la A^{-p},\qquad
\sup_{A\in\R,\,\sigma=\pm1}\int\la{A+\sigma z^2}^{-p}\,dz<\infty.
\]
\end{lemma}
\begin{proof}
For $|A|>2$, split the first integral into the two covering sets $|r|\ge|A|/2$ and $|A-r|\ge|A|/2$ and extract the corresponding weight. For bounded $A$, integrate either remaining weight. In the second integral, $y=z^2$ gives $\int_0^\infty y^{-1/2}\la{A+\sigma y}^{-p}\,dy$. The part $y<1$ is at most two, and for $y\ge1$ use $y^{-1/2}\le1$ and translation invariance of the integrable weight.
\end{proof}

\begin{proposition}\label{prop:bilinear}
For $b>1/2$,
\begin{equation}\label{eq:bilinear}
\norm{B(uv)}_{L^2_{t,x}}\le C_b\norm u_{X^{0,b}}\norm v_{X^{0,b}}.
\end{equation}
For every $s\ge0$, moreover,
\begin{equation}\label{eq:tameX}
\norm{B(uv)}_{L^2_tH^s_x}
\le C_{s,b}\bigl(\norm u_{X^{s,b}}\norm v_{X^{0,b}}
+\norm u_{X^{0,b}}\norm v_{X^{s,b}}\bigr).
\end{equation}
\end{proposition}
\begin{proof}
Let $F=\la{\tau-\xi^3}^b|\widetilde u|$ and $G=\la{\tau-\xi^3}^b|\widetilde v|$. Cauchy--Schwarz in convolution yields
\[
|\widetilde{B(uv)}(\xi,\tau)|^2\le C K_b(\xi,\tau)(F^2*G^2)(\xi,\tau),
\]
where, away from $\xi=0$,
\[
K_b=\frac{\xi^2}{1+\xi^2}\iint
\frac{d\xi_1\,d\tau_1}{\la{\tau_1-\xi_1^3}^{2b}
\la{\tau-\tau_1-(\xi-\xi_1)^3}^{2b}}.
\]
For fixed $\xi_1$, set $r=\tau_1-\xi_1^3$ and
$A=\tau-\xi_1^3-(\xi-\xi_1)^3$. The first bound in
Lemma~\ref{lem:integrals}, with $p=2b>1$, gives
\[
\int\frac{d\tau_1}{\la{\tau_1-\xi_1^3}^{2b}
\la{\tau-\tau_1-(\xi-\xi_1)^3}^{2b}}
\le C_{2b}\la{\tau-\xi_1^3-(\xi-\xi_1)^3}^{-2b}.
\]
The algebraic identity
\[
\xi_1^3+(\xi-\xi_1)^3=\frac{\xi^3}{4}
+3\xi(\xi_1-\xi/2)^2
\]
therefore implies, for $\xi\ne0$,
\begin{align*}
K_b(\xi,\tau)
&\le C_{2b}\frac{\xi^2}{1+\xi^2}
\int\la{\tau-\xi^3/4-3\xi(\xi_1-\xi/2)^2}^{-2b}\,d\xi_1\\
&=\frac{C_{2b}}{\sqrt3}\frac{|\xi|^{3/2}}{1+\xi^2}
\int\la{\tau-\xi^3/4-\operatorname{sgn}(\xi)z^2}^{-2b}\,dz
\le D_b\frac{|\xi|^{3/2}}{1+\xi^2}\le D_b.
\end{align*}
Here $z=\sqrt{3|\xi|}(\xi_1-\xi/2)$, including its Jacobian,
and $D_b$ is the product of the two constants in
Lemma~\ref{lem:integrals} divided by $\sqrt3$.
The last rational factor is at most one: use $|\xi|^{3/2}\le1$
for $|\xi|\le1$ and $|\xi|^{3/2}\le\xi^2$ otherwise.
At $\xi=0$ the double integral without its prefactor is infinite:
after integrating in $\tau_1$ its integrand is a positive constant
in $\xi_1$, because $\xi_1^3+(-\xi_1)^3=0$.
Thus the displayed formula for $K_b$ is used only for $\xi\ne0$;
we separately assign $K_b(0,\tau)=0$, avoiding a product $0\cdot\infty$.
For Schwartz inputs the multiplier itself gives
$\widetilde{B(uv)}(0,\tau)=0$. This assignment is measurable and
changes nothing in the integrated bound, since $\{\xi=0\}$ has
two-dimensional Lebesgue measure zero.

All integrands defining $K_b$ are nonnegative; Tonelli permits integration
in $\tau_1$ first even before finiteness has been established.
A second application, now to $F^2*G^2$, gives exactly
\[
\iint(F^2*G^2)(\xi,\tau)\,d\xi\,d\tau
=\norm F_2^2\norm G_2^2.
\]
Thus \eqref{eq:bilinear} holds with $C_b=C_{\rm FT}\sqrt{D_b}$,
where $C_{\rm FT}$ is the fixed convolution normalization factor.
For $s\ge0$, choose a constant $W_s$ such that
\[
\la{\eta+\zeta}^s\le W_s(\la\eta^s+\la\zeta^s);
\quad W_s=2^s\text{ is an admissible choice}.
\]
Insert this inequality in the absolute convolution, apply the preceding
proof separately with $F_s=\la\xi^sF$ and with $G_s=\la\xi^sG$,
and use the triangle inequality. This proves \eqref{eq:tameX} with
$C_{s,b}=W_s C_b$. In particular the constants depend only on $s,b$
and the Fourier convention, not on the inputs or the time interval.

Schwartz functions are dense in $X^{s,b}$: approximate the weighted
Fourier transform in $L^2$ by smooth compactly supported functions and
undo the smooth positive weights. The difference identity
$B(u_nv_n)-B(u_mv_m)=B((u_n-u_m)v_n)+B(u_m(v_n-v_m))$
shows that the images are Cauchy in $L^2_tH^s_x$, independently of the
approximating sequences. On every compact time interval $I$,
\eqref{eq:interaction} also gives
\[
\sup_{t\in I}\norm{u_n(t)v_n(t)-u(t)v(t)}_{L^1_x}
\le \sup_I\norm{u_n-u}_2\sup_I\norm{v_n}_2
+\sup_I\norm u_2\sup_I\norm{v_n-v}_2\longrightarrow0.
\]
Consequently $u_nv_n\to uv$ in distributions.
The symbol $i\xi\la\xi^{-1}$ is smooth and all its derivatives
are bounded. The Leibniz rule for Schwartz seminorms therefore shows
that its Fourier multiplier $B$ preserves $\mathcal S(\R)$ continuously.
It extends by duality to a continuous map
$B:\mathcal S'(\R)\to\mathcal S'(\R)$. Testing on
compact time intervals therefore identifies the $L^2$ limit with $B(uv)$.
This proves both the extension and its compatibility with the actual product.
\end{proof}

Choose $\chi\in C_c^\infty$, equal to one on $[-1,1]$ and supported in $[-2,2]$, and put $\chi_T(t)=\chi(t/T)$.
\begin{lemma}\label{lem:duhamel}
For $0<T\le1$, $0<b<1$ and $F\in L^2_tH^s_x$,
\begin{equation}\label{eq:duhamelbound}
\norm{\chi_T(t)\int_0^tS(t-r)F(r)\,dr}_{X^{s,b}}
\le \mathcal D_bT^{1-b}\norm F_{L^2_tH^s_x}.
\end{equation}
Here $\mathcal D_b$ depends only on $b$ and the fixed cutoff $\chi$;
it is independent of $s$ because the proof applies in the Hilbert space $H^s_x$.
Also $\norm{\chi(t)S(t)h}_{X^{s,b}}=L_b\norm h_{H^s}$,
where $L_b=\norm\chi_{H^b}$.
\end{lemma}
\begin{proof}
Set $q(r)=S(-r)F(r)$ and $h(t)=\chi_T(t)\int_0^tq(r)\,dr$. On the cutoff support Cauchy--Schwarz gives $\norm h_{L^2_tH^s}\le CT\norm F_{L^2_tH^s}$. Its derivative is
$T^{-1}\chi'(t/T)\int_0^tq+\chi_Tq$, whose $L^2_tH^s$ norm is at most $C\norm F_{L^2_tH^s}$. Interpolation yields $\norm h_{H^b_tH^s}\le CT^{1-b}\norm F_{L^2_tH^s}$. Apply \eqref{eq:interaction}. The homogeneous identity follows from $S(-t)\chi(t)S(t)h=\chi(t)h$.
\end{proof}

\section{Local and global evolution}\label{sec:cauchy}
For $I=[t_0,t_0+T]$, define $X^{s,b}(I)$ as the quotient space with norm
\[
\norm u_{X^{s,b}(I)}=\inf\{\norm U_{X^{s,b}}:U|_I=u\}.
\]
The null-restriction subspace is closed in the distribution topology, so the quotient is complete. Formula~\eqref{eq:interaction} supplies the continuous trace and the endpoint values. By almost optimal extensions, the bilinear bounds remain true on $I$ with the same constants.

\subsection{Fixed point, uniqueness and persistence}
First take $I=[0,T]$ and $Y=X^{0,b}(I)$. For $u\in Y$, let $F_u$ be the zero extension of $B(u^2)$ from $I$. Define the map on the quotient by restricting
\[
\chi(t)S(t)u_0-\chi_T(t)\int_0^tS(t-r)F_u(r)\,dr.
\]
The nonlinear term depends only on $u|_I$; no choice of extension enters this definition. On $I$, both cutoffs equal one. To keep the constants of the two estimates distinct, put
$E_b=\mathcal D_b C_b$, where $C_b$ is the constant in
\eqref{eq:bilinear}, and $E_{s,b}=\mathcal D_b C_{s,b}$,
where $C_{s,b}$ is the constant in \eqref{eq:tameX}.
Proposition~\ref{prop:bilinear} and Lemma~\ref{lem:duhamel} give
\begin{align*}
\norm{\Phi(u)}_Y&\le L_b\norm{u_0}_2+E_bT^{1-b}\norm u_Y^2,\\
\norm{\Phi(u)-\Phi(v)}_Y&\le E_bT^{1-b}(\norm u_Y+\norm v_Y)\norm{u-v}_Y.
\end{align*}
On the ball of radius $2L_b(1+\norm{u_0}_2)$, choose $T$ so the nonlinear factor is at most $1/4$. Banach's theorem yields the solution and locally Lipschitz dependence. Restarting the construction defines a maximal interval. If its finite endpoint is approached while the $L^2$ trace remains bounded, a common local time extends the solution.

Uniqueness is not restricted to the construction ball. On a shorter interval of length $h$, two solutions with the same trace satisfy the difference estimate with factor $E_bh^{1-b}$ times their restriction norms. These norms are at most their norms on the original interval. Choose $h$ small and repeat over finitely many intervals. This proves uniqueness throughout $X^{0,b}(I)$.

For $s>0$, fix a low-norm ball of radius $R_0$ on an interval
of length $h$ and run Picard iteration $u_{n+1}=\Phi(u_n)$ starting
with the homogeneous solution. Choose $h$ so that the low-norm
contraction factor $\rho$ is at most $1/2$ and
$\alpha=2E_{s,b}h^{1-b}R_0\le1/4$.
Every iterate has low norm at most $R_0$. Writing $Y_s=X^{s,b}(I)$,
\eqref{eq:tameX} gives
\[
\norm{u_{n+1}}_{Y_s}\le L_b\norm{u_0}_{H^s}
+\alpha\norm{u_n}_{Y_s},\qquad
\sup_n\norm{u_n}_{Y_s}\le M_s:=2L_b\norm{u_0}_{H^s}.
\]
For $d_n=\norm{u_n-u_{n-1}}_{Y_s}$ and
$\delta_n=\norm{u_n-u_{n-1}}_{Y_0}$, factorization and the tame bound yield
\[
d_{n+1}\le\alpha d_n+
2E_{s,b}h^{1-b}M_s\delta_n,
\qquad \delta_n\le\rho^{n-1}\delta_1.
\]
Iteration gives
$d_{n+1}\le\alpha^n d_1+
2E_{s,b}h^{1-b}M_s\delta_1\sum_{j=1}^n\alpha^{n-j}\rho^{j-1}$.
Summing over $n\ge0$ gives a finite sum because
$\alpha,\rho<1$. Thus the iterates converge strongly in $Y_s$,
and their limit is the already constructed low-regularity solution.

To justify the restriction norm on a compact interval $K$ of the
lifespan, choose finitely many open intervals $I_j$ on which the local
solution has extensions $U_j\in X^{0,b}$, with the interiors covering
$K$. Such intervals are available by restarting the local construction
at each trace. Choose smooth compactly supported temporal functions
$\theta_j$, with $\operatorname{supp}\theta_j\subset I_j$ and
$\sum_j\theta_j=1$ on a neighborhood of $K$, and set
\[
 U=\sum_j\theta_jU_j.
\]
The local solutions agree on overlaps by restriction-space uniqueness,
so $U|_K=u$. Multiplication by any $\theta\in C_c^\infty(\mathbb R)$
is bounded on $X^{0,b}$: the interaction representation gives
$S(-t)(\theta U)=\theta S(-t)U$, and the temporal Fourier convolution,
the inequality
$\langle\tau\rangle^b\le C_b\langle\tau-\sigma\rangle^b
\langle\sigma\rangle^b$, and Young's inequality yield
\[
 \|\theta U\|_{X^{0,b}}
 \le C_b\|\langle\sigma\rangle^b\widehat\theta(\sigma)\|_{L^1_\sigma}
       \|U\|_{X^{0,b}}.
\]
The $L^1$ factor is finite since $\widehat\theta$ is Schwartz.
Consequently $U\in X^{0,b}$, and the restriction norms on all
subintervals of $K$ are bounded by $\|U\|_{X^{0,b}}$.
The same argument works in $X^{s,b}$ wherever local higher regularity
has been established. A finite subdivision
with a common $h$ then propagates $H^s$ from one endpoint to the next;
the bound on each new high norm is finite, although it can grow with
the number of subdivisions. No uniform-in-time high-norm bound is
being assumed. The same difference estimate for two solutions has an
additional homogeneous term $L_b\norm{u_0-v_0}_{H^s}$.
Absorb its $\alpha$ term and use low-norm Lipschitz dependence to obtain
high-norm Lipschitz dependence on bounded sets. This also proves
persistence throughout the lower-regularity lifespan.
Uniqueness in $X^{s,b}$ follows from the $X^{0,b}$ uniqueness already
proved, since $X^{s,b}(I)\subset X^{0,b}(I)$.
Time reversal together with $x\mapsto-x$ gives the negative-time result.

We now prove unconditional uniqueness for $s\ge1/4$. The endpoint
Sobolev embedding $H^{1/4}(\mathbb R)\hookrightarrow L^4(\mathbb R)$
and $|m|\le1$ give
\[
 \|B(u^2-v^2)\|_2
 \le(\|u\|_4+\|v\|_4)\|u-v\|_4
 \le C_s(\|u\|_{H^s}+\|v\|_{H^s})\|u-v\|_{H^s}.
\]
Thus any mild solution $u\in C(I;H^s)$ has
$B(u^2)\in C(I;L^2)$. Its mild equation is meaningful as an $L^2$
Bochner integral, even when it is initially formulated in distributions.
On a subinterval $I_0=[t_0,t_0+h]\subset I$ with $0<h\le1$, extend
this force by zero outside $I_0$. After translating time, the cutoff
homogeneous solution belongs to $X^{0,b}$, and Lemma~\ref{lem:duhamel}
places the cutoff Duhamel term in $X^{0,b}$. Their difference agrees
with $u$ on $I_0$, so $u\in X^{0,b}(I_0)$.
Every $C(I;H^s)$ mild solution therefore belongs locally to the
already established uniqueness class. Applying that uniqueness on
successive subintervals proves agreement throughout $I$, in both time
directions, for solutions with the same trace. This includes the global
energy endpoint $s=1/2$ and proves the strengthened uniqueness assertion
in Theorem~\ref{thm:flow}.

For $s>1/2$, the spatial estimate
\begin{equation}\label{eq:spatial}
\norm{B(uv)}_{H^s}\le C_s\norm u_{H^s}\norm v_{H^s}
\end{equation}
follows from $|m(\xi)|\le1$, the inequality
$\langle\xi\rangle^s\le C_s(\langle\eta\rangle^s+
\langle\xi-\eta\rangle^s)$ and Young's convolution inequality:
\[
 \|B(uv)\|_{H^s}\le C_s\big(
 \|\langle\xi\rangle^s\widehat u\|_2\|\widehat v\|_1
 +\|\widehat u\|_1\|\langle\xi\rangle^s\widehat v\|_2\big).
\]
Here Cauchy--Schwarz gives
$\|\widehat u\|_1\le
(\int\langle\xi\rangle^{-2s}d\xi)^{1/2}\|u\|_{H^s}$,
and the integral is finite precisely when $s>1/2$.
Thus $u\mapsto B(u^2)$ is locally Lipschitz on $H^s$, with
$\|B(u^2-v^2)\|_{H^s}\le
C_s(\|u\|_{H^s}+\|v\|_{H^s})\|u-v\|_{H^s}$.
The Duhamel map is consequently a contraction on a ball of
$C(I;H^s)$ when $2C_s|I|R<1$.
The restriction-space solution satisfies this same mild equation,
since Proposition~\ref{prop:bilinear} identifies the quadratic force
with the actual product. The unconditional uniqueness proved above
identifies the two constructions. Restarting the spatial contraction
also gives the $H^s$ continuation criterion for $s>1/2$.
These observations finish the local part of Theorem~\ref{thm:flow}.

\subsection{Conservation and the energy endpoint}
For a smooth real solution, $JB=\partial_x$ and skew-adjointness imply
\[
\frac d{dt}\mathcal P(u)
=-\la{Ju,u_{xxx}}-\la{u,\partial_x(u^2)}
=-\frac23\int\partial_x(u^3)=0.
\]
For this computation it suffices to take real $H^5$ solutions:
$u_t\in H^2$, so differentiation of $\mathcal P$ is legitimate.
The linear term has zero real part by its odd imaginary Fourier symbol.
Moreover $u^3\in L^1$ since $u\in L^\infty\cap L^2$, and
\[
\norm{(u^3)_x}_1\le3\norm u_\infty\norm u_2\norm{u_x}_2<\infty.
\]
Thus $u^3\in W^{1,1}(\R)$; its absolutely continuous representative
has limits at both infinities, and integrability forces these limits
to be zero. Its derivative has integral zero. This argument is applied
only to the smooth approximants, not to the $H^{1/2}$ solution itself.

Choose real $u_{0,n}\in H^5$ with $u_{0,n}\to u_0$ in $H^{1/2}$,
for example smooth even Fourier cutoffs of the real datum.
Their $H^{1/2}$ norms have a common bound. The contraction at
$s=1/2$ gives one interval $[-h,h]$ and solutions satisfying
\[
\sup_{|t|\le h}\norm{u_n(t)-u(t)}_{H^{1/2}}
\le C\norm{u_{0,n}-u_0}_{H^{1/2}}\longrightarrow0.
\]
Persistence, proved using the low-norm bound above, puts each $u_n$
in $C([-h,h];H^5)$ even if its initial $H^5$ norm grows with $n$.
The preceding calculation gives $\mathcal P(u_n(t))=\mathcal P(u_{0,n})$.
Since
\[
|\mathcal P(v)-\mathcal P(w)|
\le\tfrac12(\norm v_{H^{1/2}}+\norm w_{H^{1/2}})
\norm{v-w}_{H^{1/2}},
\]
pass to the limit uniformly in $t$ to obtain conservation on this interval.
At every subsequent trace the same local argument applies.
A finite chain covers any compact interval of existence and proves
\eqref{eq:Pconservation} there without assuming global smooth solutions.

Its conserved value bounds the $L^2$ norm. The $L^2$ continuation argument therefore extends the solution for all time; persistence and the local $H^{1/2}$ construction retain the endpoint regularity. Higher $H^s$ norms remain finite on each compact interval by the tame argument, using the common lower-regularity bound. Iterating local dependence over a finite covering proves locally Lipschitz flow on each fixed compact time interval. This finishes Theorem~\ref{thm:flow}.

\subsection{Hamiltonian and a uniform higher-energy bound}
\begin{proposition}\label{prop:H}
For real $H^{3/2}$ solutions, the Hamiltonian
\begin{equation}\label{eq:H}
\mathcal H(u)=\frac12\la{Ju_x,u_x}-\frac13\int u^3
\end{equation}
is conserved. Consequently
\begin{equation}\label{eq:uniformhigher}
\sup_{t\in\R}\norm{u(t)}_{H^{3/2}}^2
\le\norm{u_0}_{H^{1/2}}^2+2|\mathcal H(u_0)|
+C\norm{u_0}_{H^{1/2}}^3.
\end{equation}
\end{proposition}
\begin{proof}
The functional is $C^1$ on the real space $H^{3/2}$, with derivative $\mathcal H'(u)=-Ju_{xx}-u^2$ in its dual. For smooth solutions $u_t=B\mathcal H'(u)$, so skew-adjointness gives $d\mathcal H/dt=0$. Smooth approximation in $C_tH^{3/2}$ passes the conservation law to the stated class. The inclusion $H^{1/2}\hookrightarrow L^3$ and \eqref{eq:Pconservation} bound the cubic integral by $C\norm{u_0}_{H^{1/2}}^3$. Finally
$\norm u_{H^{3/2}}^2=\norm u_{H^{1/2}}^2+\norm{J^{1/2}u_x}_2^2$, which proves \eqref{eq:uniformhigher}.
\end{proof}
The $L^2$ norm alone is not asserted to be conserved. Conservation of $\int u$ on the line requires additional integrability hypotheses and is not used for global continuation.

\section{Solitary waves by constrained minimization}\label{sec:waves}
For a traveling wave $\phi(x-ct)$, the integrated profile equation is
$-c\phi+\phi''+J^{-1}(\phi^2)=0$. We construct \eqref{eq:profile} first and verify localization afterwards. Its symbol satisfies
\[
\min\{c,1\}\langle\xi\rangle^3
\le a_c(\xi)=(c+\xi^2)\langle\xi\rangle
\le\max\{c,1\}\langle\xi\rangle^3,
\qquad \langle\xi\rangle=\sqrt{1+\xi^2}.
\]
Thus the profile energy space is $H^{3/2}(\R)$.

\subsection{The minimum and nonvanishing}
For real functions, define
\[
Q_c(v)=\int a_c|\widehat v|^2,
\qquad I_c=\inf\{Q_c(v):v\in H^{3/2},\ \int v^3=1\}.
\]
A normalized nonnegative test function shows $I_c<\infty$, while Sobolev embedding gives $1\le C_cQ_c(v)^{3/2}$ on the constraint, hence $I_c>0$. By homogeneity and a change of sign when needed,
\begin{equation}\label{eq:homogeneity}
Q_c(v)\ge I_c\left|\int v^3\right|^{2/3}
\end{equation}
for every real $v$.

Let $v_n$ be minimizing. It is bounded in $H^{3/2}$. On $I_j=(j,j+1)$ the local $H^1$ embedding has a uniform constant and gives
\[
\int_{I_j}|v|^3\le C\norm v_{H^1(I_j)}\norm v_{L^2(I_j)}^2.
\]
Sum and use Cauchy--Schwarz to obtain
\begin{equation}\label{eq:nonvanishing}
\norm v_3^3\le C\sup_j\norm v_{L^2(I_j)}\norm v_{H^1}\norm v_2.
\end{equation}
Let $K_1=\sup_n\norm{v_n}_{H^1}$ and $K_0=\sup_n\norm{v_n}_2$;
these finite constants are positive. Since $\norm{v_n}_3^3\ge1$,
\eqref{eq:nonvanishing} gives
\[
\sup_{j\in\mathbb Z}\norm{v_n}_{L^2(I_j)}
\ge\gamma:=(CK_1K_0)^{-1}>0.
\]
Choose $j_n$ with $\norm{v_n}_{L^2(I_{j_n})}\ge\gamma/2$;
this does not require the supremum to be attained.
Replace $v_n(x)$ by $v_n(x+j_n+1/2)$, preserving its energy and constraint.
The fixed interval $(-1/2,1/2)$ then carries squared $L^2$ mass
at least $\gamma^2/4$. After extraction,
\[
v_n\rightharpoonup v\text{ in }H^{3/2},\qquad
v_n\to v\text{ in }L^2_{\rm loc}\cap L^3_{\rm loc},\qquad v\ne0.
\]
Here is the local compactness argument in detail. For a fixed bounded
interval $K$, choose $\eta\in C_c^\infty$ equal to one on $K$.
The sequence $w_n=\eta v_n$ is bounded in $H^1$, since
$(\eta v_n)'=\eta'v_n+\eta v_n'$, and has a common compact support.
For every translation $y$ the fundamental theorem of calculus and
Minkowski's inequality give
\[
\norm{w_n(\cdot+y)-w_n}_2
\le |y|\norm{w_n'}_2\le C_\eta|y|.
\]
Thus the sequence is bounded, tight and uniformly continuous under
translations in $L^2$. Fréchet--Kolmogorov makes it precompact there.
A diagonal extraction over $K=[-j,j]$, $j\in\mathbb N$, yields local
strong $L^2$ convergence; the limit is the weak $H^{3/2}$ limit because
testing against smooth compactly supported functions identifies both.
The global $H^{3/2}$ bound gives a uniform $L^\infty$ bound for $v_n$
and for $v$. Hence
\[
\norm{v_n-v}_{L^3(K)}^3
\le \norm{v_n-v}_{L^\infty(K)}\norm{v_n-v}_{L^2(K)}^2\longrightarrow0.
\]
The fixed interval carrying positive $L^2$ mass retains it under strong
local convergence, so $v\ne0$. Weak convergence, denoted by
$\rightharpoonup$, is used at this stage; strong global $H^{3/2}$
convergence will be proved only after the splitting argument.

\subsection{Splitting and strong convergence}
The signed cubic splitting is proved directly below; see \cite{brezislieb} for the general functional-splitting principle.
Put $r_n=v_n-v$ and $q=\int v^3$. The energy inner product
is defined, with the conjugate in the second factor, by
\[
 \langle w,z\rangle_c
 =\int_{\mathbb R}a_c(\xi)\widehat{w}(\xi)
                 \overline{\widehat{z}(\xi)}\,d\xi.
\]
For real $w,z$ this integral is real; it defines a norm equivalent
to the $H^{3/2}$ norm. In particular
$\la{r_n,v}_c\to0$ by weak convergence, and expanding the square gives
\begin{equation}\label{eq:splitQ}
Q_c(v_n)=Q_c(v)+Q_c(r_n)+o(1).
\end{equation}
Expanding the cube also gives
\begin{equation}\label{eq:splitC}
1=\int v^3+\int r_n^3+o(1).
\end{equation}
Indeed, on a fixed bounded interval the cross terms tend to zero by local $L^3$ convergence. On its complement they are bounded by
$3\norm v_{L^3(\mathrm{tail})}^2\norm{r_n}_3$ and
$3\norm v_{L^3(\mathrm{tail})}\norm{r_n}_3^2$; the tail tends to zero and $r_n$ is bounded in $L^3$.

Equations~\eqref{eq:homogeneity}--\eqref{eq:splitC} imply
\[
I_c\ge Q_c(v)+I_c|1-q|^{2/3}.
\]
If $q=0$, this would read $I_c\ge Q_c(v)+I_c$, impossible since
$v\ne0$ and $a_c>0$. For $q\ne0$, \eqref{eq:homogeneity} also yields
\[
1\ge |q|^{2/3}+|1-q|^{2/3}.
\]
If $0<q<1$, each positive number $q,1-q$ is smaller than one,
so $q^{2/3}>q$ and $(1-q)^{2/3}>1-q$; their sum is greater than one.
If $q<0$ then $|1-q|>1$, and if $q>1$ then $|q|>1$,
again contradicting the inequality. The sole remaining value is $q=1$.
Now admissibility gives $Q_c(v)\ge I_c$, while the preceding upper
bound gives $Q_c(v)\le I_c$. Therefore $Q_c(v)=I_c$ and
\eqref{eq:splitQ} implies $Q_c(r_n)\to0$. Coercivity yields strong
convergence $r_n\to0$ in $H^{3/2}$. This proves compactness modulo translations for every minimizing sequence after subsequence extraction.

\subsection{Euler equation, positivity and decay}
The map $N:H^{3/2}(\R;\R)\to\R$, $N(v)=\int v^3$,
is continuously Fréchet differentiable, with
$N'(v)[h]=3\int v^2h$. Indeed the exact remainder is
$3\int vh^2+\int h^3$, bounded by
$3\norm v_3\norm h_3^2+\norm h_3^3=o(\norm h_{H^{3/2}})$;
Hölder also proves continuity of the derivative.
At the minimizer, $N'(v)[v]=3N(v)=3$, hence $N'(v)\ne0$.
The real Hilbert-space multiplier theorem gives
\[
2\operatorname{Re}\int_{\mathbb R}
 a_c(\xi)\widehat{v}(\xi)\,\overline{\widehat{h}(\xi)}\,d\xi
=3\mu\int_{\mathbb R}v(x)^2h(x)\,dx,
\qquad h\in H^{3/2}(\R;\R).
\]
Thus $2A_cv=3\mu v^2$ as an equality in $H^{-3/2}$.
Putting $h=v$ gives $2I_c=3\mu$. Consequently
$A_cv=I_cv^2$ and $\phi_c=I_cv$ satisfies
$A_c\phi_c=I_c^2v^2=\phi_c^2$ in distributions.

For every $r\in\R$, the lower symbol bound gives
\[
\norm{A_c^{-1}g}_{H^{r+3}}
\le\min\{c,1\}^{-1}\norm g_{H^r}.
\]
For $r>1/2$, $H^r$ is an algebra: the weight-distribution and Fourier
$L^1$ argument used in \eqref{eq:spatial}, without $B$, proves
$\norm{wz}_{H^r}\le C_r\norm w_{H^r}\norm z_{H^r}$.
Starting with $\phi_c\in H^{3/2}$, its square belongs to $H^{3/2}$,
and the profile equation gives $\phi_c\in H^{9/2}$.
Repeating gives $\phi_c\in H^{3/2+3j}$ for all $j\ge0$, hence in
every $H^r$, $r\ge0$. There is no circular use of smoothness in the
multiplier equation or in this bootstrap.
Its physical convolution kernel is
$K_c=G_c*G_J$, where $G_c(x)=e^{-\sqrt c|x|}/(2\sqrt c)$ and $G_J$ is the Bessel kernel. The representation
\[
(1+\xi^2)^{-1/2}=\frac1{\sqrt\pi}\int_0^\infty t^{-1/2}e^{-t}e^{-t\xi^2}\,dt
\]
shows that $G_J$ is a positive mixture of heat kernels, with integral one. Thus $K_c>0$, $K_c\in L^1$ and $\int K_c=1/c$. From
$\phi_c=K_c*\phi_c^2$ we obtain strict positivity and $\phi_c\in L^1$; Tonelli gives the integral identity in \eqref{eq:profile}.

An $H^1$ function is uniformly continuous. Combined with $L^2$ integrability this implies decay at both infinities. Apply this to the profile and its available derivatives, and differentiate the integrated equation to verify the traveling wave. This proves the existence and regularity part of Theorem~\ref{thm:waves}.

The profile is a critical point of $\mathcal H+c\mathcal P$. Its second variation is $A_c-2\phi_c$. The cubic constraint used above is not a conserved constraint of the evolution; the minimization therefore does not by itself establish orbital stability.

\subsection{The resolvent kernel and exponential localization}
The positivity used above has further consequences. We use physical
convolution kernels, so that $A_c^{-1}g=K_c*g$ without an additional
Fourier normalization factor.

\begin{lemma}\label{lem:kernelstructure}
For $c>0$, $K_c$ is bounded, continuous, positive and even, and is
strictly decreasing as a function of $|x|$. Its integral is $1/c$.
If $0<\beta<\min\{\sqrt c,1\}$, then
\begin{equation}\label{eq:kernelweighted}
 M_\beta:=\int e^{\beta|x|}K_c(x)\,dx<\infty,
 \qquad \sup_x e^{\beta|x|}K_c(x)<\infty.
\end{equation}
\end{lemma}
\begin{proof}
Let $H_t(x)=(4\pi t)^{-1/2}e^{-x^2/(4t)}$ be the heat kernel.
The Laplace representation of $(c+\xi^2)^{-1}$, combined with the
representation of $J^{-1}$, gives by Tonelli
\begin{equation}\label{eq:kernelmixture}
 K_c(x)=\frac1{\sqrt\pi}\int_0^\infty\int_0^\infty
 t^{-1/2}e^{-t-cs}H_{t+s}(x)\,ds\,dt.
\end{equation}
Every heat kernel is positive and even. If $0\le x_1<x_2$, then
$H_{t+s}(x_1)>H_{t+s}(x_2)$ for every $t,s>0$, so the same strict
inequality holds for $K_c$. Finiteness at zero also follows from
$K_c=G_c*G_J$: the resolvent kernel $G_c$ is bounded and continuous,
and $G_J\in L^1$. This convolution is bounded and continuous.
Integrating the factors gives $\int K_c=1/c$.

For the weighted estimates, $e^{\beta|x|}\le
e^{\beta|x-y|}e^{\beta|y|}$ yields
\[
 \|e^{\beta|\cdot|}K_c\|_1
 \le\|e^{\beta|\cdot|}G_c\|_1
       \|e^{\beta|\cdot|}G_J\|_1.
\]
The first factor is finite for $\beta<\sqrt c$.
Completing the square in the Gaussian gives
$\int e^{\beta|x|}H_t(x)\,dx\le2e^{\beta^2t}$.
The heat representation therefore gives
\[
 \|e^{\beta|\cdot|}G_J\|_1
 \le\frac2{\sqrt\pi}\int_0^\infty
 t^{-1/2}e^{-(1-\beta^2)t}\,dt<\infty
 \quad(\beta<1).
\]
Finally the same weighted convolution inequality, with the $L^\infty$
norm on $G_c$, proves the second assertion of
\eqref{eq:kernelweighted}.
\end{proof}

\begin{proposition}\label{prop:exponential}
Let $\phi$ be a nonzero $H^{3/2}$ solution of $A_c\phi=\phi^2$,
$c>0$. For every $0<\beta<\min\{\sqrt c,1\}$,
\[
 e^{\beta|x|}\phi\in L^2(\mathbb R),\qquad
 0<\phi(x)\le C_{\phi,\beta}e^{-\beta|x|}.
\]
\end{proposition}
\begin{proof}
The bootstrap and positivity argument above apply to any such solution:
$\phi$ is smooth, positive, bounded and tends to zero. We first obtain
a weighted $L^2$ bound without assuming it in advance. Set
$w_R(x)=e^{\beta\min\{|x|,R\}}$. Since the truncated distance is
one-Lipschitz,
$w_R(x)\le e^{\beta|x-y|}w_R(y)$.
The integral equation and Young's inequality give
\[
 \|w_R\phi\|_2\le M_\beta\|w_R\phi^2\|_2.
\]
Choose $L$ so large that $\delta:=\sup_{|y|>L}\phi(y)$ satisfies
$M_\beta\delta<1/2$. Splitting the square into its compact and tail
parts yields
\[
 \|w_R\phi^2\|_2
 \le e^{\beta L}\|\phi^2\|_2+\delta\|w_R\phi\|_2.
\]
Thus $\|w_R\phi\|_2\le2M_\beta e^{\beta L}\|\phi^2\|_2$ uniformly
in $R$. Monotone convergence gives $e^{\beta|x|}\phi\in L^2$.
In particular $\int e^{\beta|y|}\phi(y)^2dy<\infty$.
Using the weighted supremum of the kernel and
$|x|\le|x-y|+|y|$ now gives
\[
 e^{\beta|x|}\phi(x)
 \le\sup_z e^{\beta|z|}K_c(z)
       \int e^{\beta|y|}\phi(y)^2dy,
\]
which proves the pointwise bound.
\end{proof}
The admissible exponents reflect the two factors of the resolvent:
the speed-dependent scale $c^{-1/2}$ and the fixed Bessel scale.
No assertion of the optimal tail coefficient or endpoint exponent is made.
This localization differs from the algebraic BO profile in the comparison
above, although the high-frequency nonlinear multiplier here is bounded.
Profile tails depend on the full resolvent, not only its high-frequency order.

\subsection{Symmetry and strict monotonicity}
The following proof concerns the positive integral equation rather than
rearrangement of a form of order three. In particular, reflection
invariance of the minimization is not being used as a symmetry theorem.
The integral moving-plane strategy is classical \cite{chenliou}; the
kernel assumptions needed for this version have just been verified.

\begin{theorem}\label{thm:symmetry}
Every nonzero real $H^{3/2}$ solution of $A_c\phi=\phi^2$, $c>0$, is
symmetric about a unique point $x_0$ and is strictly decreasing with
distance from $x_0$. In particular $\phi(x_0+\cdot)$ is even and
strictly decreasing on $(0,\infty)$.
\end{theorem}
\begin{proof}
For $\lambda\in\mathbb R$, write $x^\lambda=2\lambda-x$,
$\Sigma_\lambda=(-\infty,\lambda)$ and
$d_\lambda(x)=\phi(x)-\phi(x^\lambda)$ on $\Sigma_\lambda$.
Split the convolution into the two half-lines and reflect the second
one. Evenness of $K_c$ gives the exact identity
\begin{equation}\label{eq:reflectionidentity}
 d_\lambda(x)=\int_{\Sigma_\lambda}
 [K_c(x-y)-K_c(x^\lambda-y)]
 [\phi(y)^2-\phi(y^\lambda)^2]dy.
\end{equation}
For $x,y<\lambda$,
$|x-y|<|x^\lambda-y|$, so the bracket involving $K_c$ is strictly
positive. All integrals are legitimate: $K_c$ is bounded and
$\phi^2\in L^1$, and the reflection preserves integrability.

Let $E_\lambda=\{y<\lambda:d_\lambda(y)>0\}$ and
$D_\lambda=\sup_{x<\lambda}(d_\lambda(x))_+$.
Discard the nonpositive contribution outside $E_\lambda$ in
\eqref{eq:reflectionidentity}. On $E_\lambda$,
$\phi(y^\lambda)<\phi(y)$, hence
\begin{equation}\label{eq:planebound}
 D_\lambda\le
 2D_\lambda\sup_{x<\lambda}
 \int_{E_\lambda}K_c(x-y)\phi(y)\,dy.
\end{equation}
The supremum is finite since $\phi$ is bounded and $K_c\in L^1$.
If $\lambda$ is sufficiently negative, then
$2\|K_c\|_1\sup_{y<\lambda}\phi(y)<1$.
Equation~\eqref{eq:planebound} forces $D_\lambda=0$.
Thus $\phi(x)\le\phi(x^\lambda)$ starts on the far left.

Let $\lambda_0$ be the supremum of the positions up to which this
inequality holds for every preceding plane. It is finite: if it held
for all planes, then for any $x_1<x_2$, taking
$\lambda=(x_1+x_2)/2$ would make $\phi(x_1)\le\phi(x_2)$;
a positive function tending to zero at $+\infty$ cannot have this
property. Continuity gives $d_{\lambda_0}\le0$.

If $d_{\lambda_0}$ is not identically zero, then it is strictly
negative at every point of $\Sigma_{\lambda_0}$. Indeed, it is negative
on some open interval, and the positive kernel difference in
\eqref{eq:reflectionidentity} propagates the strict inequality to
every $x<\lambda_0$.
We show that strict inequality permits further movement.
Choose $R$ large with $-R<\lambda_0-1$ and
\[
 2\|K_c\|_1\sup_{y<-R}\phi(y)<\tfrac14.
\]
Then choose $0<\epsilon<1$ so that
$4\epsilon\|K_c\|_\infty\|\phi\|_\infty<1/4$.
On the compact interval $[-R,\lambda_0-\epsilon]$, strict negativity
is uniform. Uniform continuity of the reflected values therefore gives
$d_\mu<0$ there for every $\mu>\lambda_0$ sufficiently close to
$\lambda_0$, with $\mu<\lambda_0+\epsilon$.
Consequently
\[
 E_\mu\subset(-\infty,-R)\cup(\lambda_0-\epsilon,\mu),
\]
and the second interval has length at most $2\epsilon$.
The right side coefficient in \eqref{eq:planebound} is now at most
\[
 2\|K_c\|_1\sup_{y<-R}\phi(y)
 +4\epsilon\|K_c\|_\infty\|\phi\|_\infty<\tfrac12.
\]
Thus $D_\mu=0$ for all these planes, contradicting the definition of
$\lambda_0$. We conclude $d_{\lambda_0}=0$, which proves symmetry.

For $\lambda<\lambda_0$, the inequality is strict. Otherwise the
strictness alternative just proved would give a second symmetry center.
Two different reflection centers would make $\phi$ periodic, impossible
for a nonzero function tending to zero at infinity. Given
$x_1<x_2\le\lambda_0$, the midpoint plane therefore gives
$\phi(x_1)<\phi(x_2)$. Reflection proves strict decrease on the right.
It also makes $\lambda_0$ the unique point of maximum, establishing
uniqueness of the center but not uniqueness of the profile.
\end{proof}
The existence, symmetry and localization assertions of
Theorem~\ref{thm:waves} are now proved.

\subsection{A small-speed KdV limit for minimizing profiles}
The low-frequency expansion of the Bessel filter suggests a KdV limit
for small-speed profiles. We prove this limit variationally, without
assuming uniqueness of the Bessel profiles; it then supplies the base
profile for the stability perturbation.
Throughout this subsection profiles are centered at their unique maximum.

\begin{theorem}\label{thm:kdvlimit}
Let $\phi_c$ be any centered minimizing profile constructed above. Then
\begin{equation}\label{eq:kdvlimit}
 q_c(y):=c^{-1}\phi_c(y/\sqrt c)
 \longrightarrow q_0(y):=\tfrac32\operatorname{sech}^2(y/2)
 \quad\text{strongly in }H^1(\mathbb R),\qquad c\downarrow0.
\end{equation}
The minimum and quadratic invariant have the asymptotics
\begin{equation}\label{eq:smallasymptotics}
 I_c^3\sim\tfrac{36}{5}c^{5/2},\qquad
 \mathcal P(\phi_c)\sim3c^{3/2}.
\end{equation}
The conclusions hold for every choice of centered minimizers, without
an assumed smooth branch in $c$.
\end{theorem}
\begin{proof}
For $v(x)=c^{1/6}w(\sqrt c\,x)$, direct changes of variables give
\[
 \int v^3=\int w^3,\qquad
 Q_c(v)=c^{5/6}\mathcal Q_c(w),
\]
where
\[
 \mathcal Q_c(w)=\int(1+\eta^2)\sqrt{1+c\eta^2}
                         |\widehat w(\eta)|^2d\eta.
\]
Consequently $I_c=c^{5/6}\mathcal I_c$, where $\mathcal I_c$ is
the minimum of $\mathcal Q_c$ on the unit cubic constraint in
$H^{3/2}$. Define also the limiting minimum
\[
 \mathcal I_0=\inf\{\|w\|_{H^1}^2:w\in H^1(\mathbb R;\mathbb R),
                                      \ \int w^3=1\}.
\]
It is positive by $H^1\hookrightarrow L^3$ and finite by a test function.
The nonvanishing and splitting proof in Section~\ref{sec:waves} applies
with $Q_c$ replaced by $\|\cdot\|_{H^1}^2$: a minimizing sequence is
bounded in $H^1$, the interval estimate \eqref{eq:nonvanishing} prevents
vanishing, local compactness gives strong local $L^2$ and $L^3$
convergence, and the two splittings force the limiting cubic mass to
equal one. Thus the limiting minimum is attained, and its minimizing
sequences are compact in $H^1$ modulo translations.

For completeness, its minimizer can be identified uniquely. The multiplier
equation and rescaling give $q=\mathcal I_0w$ satisfying
$-q''+q=q^2$. Inversion of $1-\partial_y^2$ makes $q$ strictly positive,
and the one-dimensional algebra bootstrap makes it smooth. The $H^1$
bound and the equation give decay of $q$ and $q'$ at infinity.
Multiplication by $q'$ and integration yield
\[
 (q')^2=q^2-\tfrac23q^3,
\]
with integration constant zero by this decay. A maximum exists, and
the identity gives its value $3/2$. The second-order ODE with initial
data $q(0)=3/2$, $q'(0)=0$ has a unique solution; direct substitution
identifies it with $q_0$, up to translation. Since
\[
 \int q_0^3=\tfrac{36}{5},\qquad
 \|q_0\|_2^2=6,
\]
where $z=\tanh(y/2)$ gives
$\int q_0^3=(27/4)\int_{-1}^1(1-z^2)^2dz
=(27/4)(16/15)=36/5$,
and $Q_0(q_0)=\int q_0^3$, normalization yields
$\mathcal I_0=(36/5)^{1/3}$ and the centered minimizer is
$w_0=q_0/\mathcal I_0$.

The pointwise inequality $\sqrt{1+c\eta^2}\ge1$ gives
$\mathcal I_c\ge\mathcal I_0$. Conversely the smooth minimizer
$w_0$ is admissible for every $c>0$. For $0<c\le1$,
the integrand in $\mathcal Q_c(w_0)$ is dominated by
$(1+\eta^2)^{3/2}|\widehat w_0|^2$, which is integrable.
Dominated convergence gives
\[
 \mathcal I_0\le\mathcal I_c\le\mathcal Q_c(w_0)
 \longrightarrow\mathcal I_0.
\]
Let $w_c$ be the dilation of the normalized minimizer
$v_c=\phi_c/I_c$. Then $\int w_c^3=1$, and
\[
 \mathcal I_0\le\|w_c\|_{H^1}^2
 \le\mathcal Q_c(w_c)=\mathcal I_c\longrightarrow\mathcal I_0.
\]
Thus along every sequence $c_n\downarrow0$, $w_{c_n}$ is a minimizing
sequence for the limiting problem. After translations and extraction
it converges strongly in $H^1$ to a translate of $w_0$.

We remove those translations using the already proved shape information.
Each $w_c$ is positive and has its unique maximum at zero. Strong
$H^1$ convergence of translated functions implies uniform convergence,
since $H^1\hookrightarrow L^\infty$. Their maximum positions cannot
escape to infinity: the nonzero limiting maximum persists, whereas the
limit is uniformly small outside a fixed bounded interval. Hence the
translation parameters are bounded. Passing to a further subsequence
gives strong convergence of the unshifted $w_c$. The limit has a maximum
at zero by uniform convergence, and the limiting ODE classification
therefore identifies it with the centered $w_0$. As every sequence has
a subsequence with this same limit, the entire family satisfies
$w_c\to w_0$ in $H^1$.

The dilations give $q_c=\mathcal I_c w_c$, proving
\eqref{eq:kdvlimit}. Cubing $I_c=c^{5/6}\mathcal I_c$ gives the first
asymptotic formula. For the second, Fourier scaling gives exactly
\[
 2\mathcal P(\phi_c)=c^{3/2}
 \int\sqrt{1+c\eta^2}|\widehat q_c(\eta)|^2d\eta.
\]
Since $0\le\sqrt{1+c\eta^2}-1\le c\eta^2/2$, the difference of
this integral from $\|q_c\|_2^2$ is at most
$\tfrac c2\|q_c'\|_2^2$, which tends to zero by the $H^1$ convergence.
The limit is $\|q_0\|_2^2=6$, yielding
$\mathcal P(\phi_c)\sim3c^{3/2}$.
\end{proof}
These asymptotics do not determine the sign of
$d\mathcal P(\phi_c)/dc$ on an arbitrary branch. An asymptotic formula
for a function alone cannot be differentiated without further control.
The small-speed verification below adds derivative control before drawing a slope conclusion; outside that regime the hypothesis is retained.

\section{Verification of the stability hypotheses and orbital stability}\label{sec:stability}
Fix a minimizing profile $\phi=\phi_c$, centered so that it is even,
and let
\[
 \mathcal S_c=\mathcal H+c\mathcal P,\qquad
 L_c=A_c-2\phi.
\]
We work in the real energy space $X=H^{3/2}(\mathbb R)$.
The form of $L_c$ is continuous on $X$, with associated self-adjoint
operator on $L^2$ of domain $H^3$. A bounded potential is sufficient
for this domain statement, since $A_c$ is self-adjoint with domain
$H^3$ and $\phi\in L^\infty$.

\subsection{The negative index and the translation mode}
\begin{proposition}\label{prop:index}
The essential spectrum of $L_c$ is $[c,\infty)$. It has exactly one
negative eigenvalue, counted with multiplicity, and
$L_c\phi'=0$. In particular the negative eigenvalue is simple.
These assertions do not exclude additional zero modes.
\end{proposition}
\begin{proof}
The symbol $a_c(\xi)$ is continuous, has minimum $c$, and increases
to infinity with $|\xi|$, so $A_c$ has spectrum $[c,\infty)$.
Multiplication by $\phi$ is compact from $H^3$ to $L^2$: for a
bounded sequence in $H^3$, local Rellich compactness gives strong
$L^2$ convergence on each bounded interval after extraction, whereas
the tail is bounded by
$\sup_{|x|>R}|\phi(x)|$ times the global $L^2$ bound.
The latter supremum tends to zero. Therefore
$\phi(A_c+i)^{-1}$ is compact. The essential spectrum is unchanged
under this relatively compact self-adjoint perturbation, and all
spectral points below $c$ are isolated eigenvalues of finite multiplicity.

The constrained minimum gives nonnegativity of the form on
\[
 T=\{h\in X:\int\phi^2h=0\}.
\]
To verify this explicitly, put $N_0=\int\phi^3=Q_c(\phi)>0$ and
\[
 \gamma(t)=\left(\frac{N_0}{\int(\phi+th)^3}\right)^{1/3}
              (\phi+th),\qquad h\in T.
\]
For small real $t$ the denominator is positive, and $\gamma$ lies
on the constraint $\int\gamma^3=N_0$.
Homogeneity of the original minimization makes $\phi$ a minimum on
this constraint. Write $n(t)=\int(\phi+th)^3$ and $a(t)=(N_0/n(t))^{1/3}$,
so that $\gamma(t)=a(t)(\phi+th)$.
The tangent condition gives $n'(0)=3\int\phi^2h=0$, whereas
$n''(0)=6\int\phi h^2$. Differentiating the scalar prefactor yields
$a(0)=1$, $a'(0)=0$ and $a''(0)=-2\int\phi h^2/N_0$.
Thus $\gamma'(0)=h$ and $\gamma''(0)=a''(0)\phi$. Since
$Q_c(\phi)=N_0$, the quadratic form gives
\[
 \left.\frac{d^2}{dt^2}Q_c(\gamma(t))\right|_{t=0}
 =2Q_c(h)+2a''(0)Q_c(\phi)
 =2Q_c(h)-4\int\phi h^2.
\]
Hence
\[
 0\le\left.\frac{d^2}{dt^2}Q_c(\gamma(t))\right|_{t=0}
 =2\langle L_ch,h\rangle.
\]
Since $T$ has codimension one, a two-dimensional negative spectral
subspace would meet $T$ nontrivially, contradicting this inequality.
There is at least one negative direction because
\[
 \langle L_c\phi,\phi\rangle=-\int\phi^3<0.
\]
The spectral theorem gives exactly one negative eigenvalue.
Finally differentiation of $A_c\phi=\phi^2$, justified by smoothness,
gives $L_c\phi'=0$. Since $\phi$ is nonconstant, $\phi'\ne0$.
\end{proof}

\subsection{The branch and the slope condition}
Assume at a fixed speed $c_0>0$ that
\begin{equation}\label{eq:ND}
 \ker L_{c_0}=\operatorname{span}\{\phi_{c_0}'\}.
\end{equation}
For now this is a hypothesis, not a consequence of Proposition~\ref{prop:index}. It will be verified at sufficiently small speeds in Theorem~\ref{thm:smallspeedstability}.
Since the centered profile is even, its derivative is odd. The kernel
is thus trivial on the even subspace. On that subspace
$L_{c_0}:H^3_{\rm even}\to L^2_{\rm even}$ is a bounded bijection:
Fredholm index zero follows from the compact perturbation of the
invertible map $A_{c_0}$, and injectivity follows from
\eqref{eq:ND}. Its inverse is bounded.
The map
\[
 F:(0,\infty)\times H^3_{\rm even}\longrightarrow L^2_{\rm even},
 \qquad F(c,v)=J(c-\partial_x^2)v-v^2
\]
is continuously differentiable, since $H^3$ is an algebra and the
linear symbol has order three. The implicit function theorem therefore
gives a locally unique $C^1$ even branch $c\mapsto\phi_c$ through
$\phi_{c_0}$. This is uniqueness near the fixed profile in the even
$H^3$ space; it is not a classification of all profiles at that speed.
Differentiating the equation gives
\begin{equation}\label{eq:slopeidentity}
 L_{c_0}\dot\phi=-J\phi,\qquad
 p':=\left.\frac{d}{dc}\mathcal P(\phi_c)\right|_{c=c_0}
 =\langle J\phi,\dot\phi\rangle,
\end{equation}
where $\phi=\phi_{c_0}$ and $\dot\phi$ is even. In particular
$\langle\dot\phi,\phi'\rangle=0$.
The second hypothesis of the general stability criterion is
\begin{equation}\label{eq:positiveslope}
 p'>0.
\end{equation}
No sign of this derivative is assumed outside the theorem below.

\begin{lemma}\label{lem:coercivity}
Under \eqref{eq:ND}--\eqref{eq:positiveslope}, there is $a>0$ such that
\[
 \langle L_{c_0}z,z\rangle\ge a\|z\|_X^2
 \quad\text{if }\langle J\phi,z\rangle=\langle\phi^{\prime},z\rangle=0.
\]
For sufficiently large $\eta>0$, there is $a_\eta>0$ such that
\begin{equation}\label{eq:penaltycoercivity}
 \langle L_{c_0}h,h\rangle+2\eta\langle J\phi,h\rangle^2
 \ge a_\eta\|h\|_X^2
 \quad\text{if }\langle\phi',h\rangle=0.
\end{equation}
\end{lemma}
\begin{proof}
Write $g=J\phi$, $\zeta=\dot\phi$ and
$q(h,k)=\langle L_{c_0}h,k\rangle$.
Then $q(\zeta,\zeta)=-p'<0$ and
$q(\zeta,z)=-\langle g,z\rangle=0$ for $z\perp g$.
If this vector $z$ were a negative direction, its form-orthogonality to
$\zeta$ would produce a two-dimensional negative subspace, contrary
to Proposition~\ref{prop:index}. Thus $q$ is nonnegative on $g^\perp$.

Let $Z=\{z\in X:\langle g,z\rangle=\langle\phi',z\rangle=0\}$.
If $z\in Z$ and $q(z,z)=0$, nonnegativity and the discriminant of
$q(z+tw,z+tw)$ imply $q(z,w)=0$ for every $w\in Z$.
Thus $L_{c_0}z=\alpha g+\beta\phi'$ in $X^*$.
Pairing with $\phi'$ gives $\beta=0$, since
$L_{c_0}\phi'=0$ and $\langle g,\phi'\rangle=0$ by parity.
The equation first gives $z\in H^3$, since $A_{c_0}z=2\phi z+\alpha g\in L^2$. The identity $L_{c_0}\zeta=-g$ and \eqref{eq:ND} give
$z=-\alpha\zeta+d\phi'$.
Pairing first with $g$ and then with $\phi'$ gives
$\alpha p'=0$ and $d\|\phi'\|_2^2=0$, hence $z=0$.

Strict positivity becomes uniform coercivity by compactness of the
potential term. Otherwise choose $z_n\in Z$ with $\|z_n\|_X=1$
and $q(z_n,z_n)\to0$. After extraction $z_n\rightharpoonup z$ in
$X$ and locally strongly in $L^2$. Decay of $\phi$ implies
$\int\phi z_n^2\to\int\phi z^2$ by the same compact/tail splitting.
Weak lower semicontinuity of $Q_{c_0}$ gives $q(z,z)\le0$.
The preceding positivity forces $z=0$, so the potential integrals tend
to zero. Then $Q_{c_0}(z_n)\to0$, contradicting its equivalence to
$\|z_n\|_X^2$. This proves the first assertion.

For $h\perp\phi'$, set $b=\langle g,h\rangle/p'$ and
$z=h-b\zeta$. Then $z\in Z$ and the cross term vanishes, so
\[
 q(h,h)+2\eta\langle g,h\rangle^2
 =q(z,z)+b^2p'(2\eta p'-1).
\]
Choose $\eta>1/(2p')$. The right side controls
$\|z\|_X^2+b^2$, and hence $\|h\|_X^2$, proving
\eqref{eq:penaltycoercivity}.
\end{proof}

\subsection{The Lyapunov proof}
\begin{theorem}\label{thm:conditionalstability}
Let $\phi_{c_0}$ be a minimizing profile at $c_0>0$. If
\eqref{eq:ND} and \eqref{eq:positiveslope} hold, its translation orbit
$\mathcal O=\{\phi_{c_0}(\cdot-y):y\in\mathbb R\}$ is stable in
$H^{3/2}$ for the real flow. Precisely, for every
$\epsilon>0$ there is $\delta>0$ such that
\[
 \inf_y\|u_0-\phi_{c_0}(\cdot-y)\|_{H^{3/2}}<\delta
 \ \Longrightarrow\
 \sup_{t\in\mathbb R}\inf_y
 \|u(t)-\phi_{c_0}(\cdot-y)\|_{H^{3/2}}<\epsilon.
\]
This is a sufficient criterion for arbitrary speeds. Theorem~\ref{thm:smallspeedstability} verifies its hypotheses for all sufficiently small positive speeds.
\end{theorem}
\begin{proof}
The global flow in $X$ is supplied by Theorem~\ref{thm:flow}; both
$\mathcal H$ and $\mathcal P$ are conserved by Proposition~\ref{prop:H}.
Define the translation-invariant functional
\[
 \mathcal F(u)=\mathcal S_{c_0}(u)-\mathcal S_{c_0}(\phi)
             +\eta[\mathcal P(u)-\mathcal P(\phi)]^2.
\]
It is conserved along the flow. Since $\mathcal S_{c_0}'(\phi)=0$,
its first variation vanishes at $\phi$, and its second variation is
exactly the form in \eqref{eq:penaltycoercivity}.
More explicitly, the action remainder is $-\frac13\int h^3$;
in the penalty we use
$\mathcal P(\phi+h)-\mathcal P(\phi)=\langle g,h\rangle+
\frac12\|h\|_{H^{1/2}}^2$.
Sobolev embedding bounds all terms beyond the quadratic part by
$C(\|h\|_X^3+\|h\|_X^4)$. Thus, when $h\perp\phi'$ and
$\|h\|_X$ is small,
\begin{equation}\label{eq:localLyapunov}
 \mathcal F(\phi+h)\ge k\|h\|_X^2
\end{equation}
for some $k>0$. Continuity of the quadratic form and the same expansion
also give $|\mathcal F(\phi+h)|\le K\|h\|_X^2$ for all sufficiently
small $h$, without an orthogonality requirement.

We explain the removal of the translation direction. For $u$ close
to $\phi$ in $X$, solve
\[
 \Psi(u,y):=\langle u(\cdot+y)-\phi,\phi'\rangle=0.
\]
Equivalently $\Psi(u,y)=\langle u,\phi'(\cdot-y)\rangle-
\langle\phi,\phi'\rangle$, so the map is $C^1$ even for
$u\in X$. At $(\phi,0)$ its $y$ derivative is
$\|\phi'\|_2^2>0$. The implicit function theorem gives a small shift
$y(u)$ and a remainder $h=u(\cdot+y(u))-\phi$ satisfying
$h\perp\phi'$. Moreover $|y(u)|\le C\|u-\phi\|_X$.
Since $\phi$ is smooth, translation of $\phi$ is locally Lipschitz
in $X$, so $\|h\|_X\le C\|u-\phi\|_X$.
The construction is uniform near every translate by invariance.

Let $d(u,\mathcal O)=\inf_y\|u-\phi(\cdot-y)\|_X$.
When this distance is small, choose a translate within twice the
distance (the zero-distance case belongs to the closed orbit).
The modulation just described and \eqref{eq:localLyapunov} give
\[
 k_0d(u,\mathcal O)^2\le\mathcal F(u)
 \le K_0d(u,\mathcal O)^2
\]
in a fixed tubular neighborhood of the orbit, with the upper bound
understood from the absolute-value estimate. The orbit is closed:
unbounded translations tend weakly to zero in $X$ and cannot converge
strongly to a translate of its nonzero profile; bounded translations
have convergent subsequences.

The distance to any subset is one-Lipschitz, so
$t\mapsto d(u(t),\mathcal O)$ is continuous. Choose a radius $r$
inside the neighborhood and less than $\epsilon$. If the initial
distance is $\delta$ with $K_0\delta^2<k_0r^2$, a first exit at
distance $r$ would contradict conservation of $\mathcal F$.
The argument applies in both time directions and the flow is global.
Thus the distance remains below $r<\epsilon$ for all $t$.
\end{proof}
The proof separates the general spectral criterion from its verification. Symmetry and the single negative direction are established at every positive speed; the following perturbation supplies the kernel and slope conditions at small speed.
The cubic minimization constraint itself is not conserved by the
evolution; treating it as such would give an invalid stability proof.

\subsection{Verification of the stability conditions at small speed}
The operator $L_c$ is nonlocal even when restricted to even functions.
Parity therefore does not turn its equation into a Sturm--Liouville
boundary-value problem. We instead perturb the integral equation after
the KdV rescaling. The limiting equation is local, and its one-dimensional
kernel can be checked directly by a Wronskian argument.

Put $Y=H^1(\mathbb R;\mathbb R)$ and, for $c\ge0$, define
\[
 R_c=(1-\partial_y^2)^{-1}(1-c\partial_y^2)^{-1/2},\qquad
 T_c=\operatorname{Id}-2R_c(q_c\,\cdot).
\]
Here $q_c$ is the rescaled centered minimizing profile of
Theorem~\ref{thm:kdvlimit}, and
$q_0=\tfrac32\operatorname{sech}^2(y/2)$.
The rescaled profile equation is
\begin{equation}\label{eq:rescaledfixedpoint}
 q_c=R_c(q_c^2).
\end{equation}
At zero speed this is the integrated KdV equation. The useful feature
is operator-norm convergence on the same space $Y$:
\begin{equation}\label{eq:resolventperturbation}
 \|R_c\|_{\mathcal L(Y)}\le1,\qquad
 \|R_c-R_0\|_{\mathcal L(Y)}\le c/2.
\end{equation}
Indeed the difference symbol is bounded in modulus by
$c\eta^2/[2(1+\eta^2)]$. To clarify the operator norm in
\eqref{eq:resolventperturbation}, a scalar Fourier multiplier $T_m$
with bounded symbol $m$ satisfies
\[
 \|T_m\|_{\mathcal L(H^1,H^1)}
 =\operatorname*{ess\,sup}_{\eta\in\mathbb R}|m(\eta)|.
\]
Indeed, the same Sobolev weight occurs in the input and output norms:
\[
 \|T_mf\|_{H^1}^2
 =\int_{\mathbb R}(1+\eta^2)|m(\eta)|^2|\widehat f(\eta)|^2\,d\eta
 \le\|m\|_\infty^2\|f\|_{H^1}^2.
\]
The reverse inequality follows by choosing $\widehat f$ supported on
bounded positive-measure sets where $|m|$ approaches its essential
supremum. The same norm identity holds on the real space $Y$ for
the real even symbols used here, by choosing real even Fourier data.
Thus no additional factor $\langle\eta\rangle$ occurs in the
$H^1\to H^1$ norm. Such a factor would instead give the multiplier
norm from $L^2$ to $H^1$. In particular,
\[
 \|R_c-R_0\|_{\mathcal L(Y,Y)}
 =\sup_{\eta\in\mathbb R}
 \frac{1-(1+c\eta^2)^{-1/2}}{1+\eta^2}
 \le\frac c2,
\]
using $0\le1-(1+z)^{-1/2}\le z/2$ for $z\ge0$.
We perturb this smoothing resolvent, rather
than the unbounded symbol $J_c(1-\partial_y^2)$.

\begin{lemma}\label{lem:limitingkernel}
The operator $T_0$ is Fredholm of index zero on $Y$, with
$\ker T_0=\operatorname{span}\{q_0'\}$. Its restriction to $Y_{\rm even}$
is a bounded bijection. There is $C_0>0$ such that
\begin{equation}\label{eq:limitinggap}
 \|h\|_Y\le C_0\|T_0h\|_Y
 \quad\text{whenever }\langle h,q_0'\rangle_{L^2}=0.
\end{equation}
\end{lemma}
\begin{proof}
Multiplication by $q_0$ is compact from $H^1$ to $L^2$: local Rellich
compactness controls a bounded interval, and decay of $q_0$ controls
its complement. Since $R_0:L^2\to H^2$ is bounded, $R_0(q_0\,\cdot)$
is compact on $Y$. Thus $T_0$ is identity minus a compact map, of
Fredholm index zero. The same assertion holds on its invariant even
subspace by the identical compactness argument.

If $T_0h=0$, then $h\in H^2$ and
\[
 -h''+(1-2q_0)h=0.
\]
Differentiating the KdV profile equation gives this equation for $q_0'$.
Both $h,h'$ and $q_0',q_0''$ tend to zero at infinity. Their Wronskian
$h q_0''-h'q_0'$ is constant by the two differential equations,
and its limit at $+\infty$ is zero. On $(0,\infty)$, $q_0'$ never
vanishes; hence $(h/q_0')'=0$ there and $h=a q_0'$ for a constant $a$.
At zero, continuity gives $h(0)=0$ and $h'(0)=a q_0''(0)$;
since $q_0''(0)=-3/4\ne0$, the proportionality constant is fixed
by these initial data. Uniqueness for the second-order ODE therefore
extends $h=a q_0'$ across zero to the entire line.
Thus the kernel is exactly the stated odd translation mode, and the
even restriction is injective. Its index zero makes it surjective,
with bounded inverse by the open mapping theorem.

Finally a Fredholm operator has closed range and induces a bounded
inverse from its range to the quotient by its kernel. Taking the closed
complement $\langle h,q_0'\rangle=0$ gives
\eqref{eq:limitinggap}. Equivalently, a sequence of unit vectors in
that complement with $T_0h_n\to0$ would have a strongly convergent
subsequence by compactness of $\operatorname{Id}-T_0$, whose limit
would be a nonzero kernel vector in the complement, a contradiction.
\end{proof}

\begin{lemma}\label{lem:smallbranch}
For some $c_*>0$, the centered minimizing profiles are unique for
$0<c<c_*$ and their rescalings obey
\begin{equation}\label{eq:branchbounds}
 \|q_c-q_0\|_{H^1}\le C\cdot c,\qquad
 q_c\in C^1((0,c_*);H^1),\qquad
 \|\partial_cq_c\|_{H^1}\le C.
\end{equation}
Uniqueness here concerns centered minimizers, and more generally even
solutions in a fixed $H^1$ neighborhood of $q_0$ after rescaling.
The constants $C$ in \eqref{eq:branchbounds} are uniform for
$0<c<c_*$. They may depend on the chosen threshold $c_*$, the fixed
model normalization, the $H^1$ algebra constant and the limiting inverse
bound, but not on the individual speed in that interval.
\end{lemma}
\begin{proof}
Let $\mathcal E_0=T_0|_{Y_{\rm even}}$, which has a bounded inverse by
Lemma~\ref{lem:limitingkernel}. Write $q=q_0+h$. Since $Y$ is an
algebra, the equation $q=R_cq^2$ is equivalent to
\begin{equation}\label{eq:branchcontraction}
 h=\mathcal E_0^{-1}\big[(R_c-R_0)q_0^2
           +2(R_c-R_0)(q_0h)+R_c(h^2)\big].
\end{equation}
On a fixed sufficiently small ball $\|h\|_Y\le\rho$, the difference
of the right sides for $h,k$ is bounded by
$C(c+\rho)\|h-k\|_Y$, using \eqref{eq:resolventperturbation} and
the algebra estimate. Choose $\rho$ and then $c_*$ so this constant
is below $1/2$. The value at $h=0$ has norm at most $C\cdot c$.
Shrinking $c_*$ again makes the map preserve the ball and gives a
unique even fixed point. Comparing this fixed point with zero in
\eqref{eq:branchcontraction} yields $\|h\|_Y\le C\cdot c$.
The convergence in Theorem~\ref{thm:kdvlimit} puts every centered
minimizer in this ball for all sufficiently small $c$. Otherwise a
sequence of excluded minimizers at speeds tending to zero would
contradict that theorem. Thus all centered minimizers coincide with
the fixed point just constructed.

For each positive $c$, $R_c$ is continuously differentiable in operator
norm, with derivative symbol
\[
 r_c'(\eta)=-\frac{\eta^2}{2(1+\eta^2)(1+c\eta^2)^{3/2}},
 \qquad \|R_c'\|_{\mathcal L(Y,Y)}\le\tfrac12.
\]
Here $\mathcal L(Y)=\mathcal L(Y,Y)$ denotes the bounded operators on
$Y=H^1(\mathbb R;\mathbb R)$; the displayed norm is an operator
norm, not an $L^2$ norm of the symbol. By the same multiplier-norm
identity used after \eqref{eq:resolventperturbation},
\[
 \|R_c'\|_{\mathcal L(Y,Y)}
 =\sup_{\eta\in\mathbb R}
 \frac{\eta^2}{2(1+\eta^2)(1+c\eta^2)^{3/2}}
 \le\tfrac12.
\]
Again, multiplying the symbol by $\langle\eta\rangle$ would measure
an $L^2\to H^1$ norm, which is not the norm required here.
Local operator-norm continuity of the derivative follows, for example,
by bounding its second derivative on any compact positive-speed
interval. No operator-norm differentiability at $c=0$ is needed.
The equation derivative on even $Y$ is
$\mathcal E_c=\operatorname{Id}-2R_c(q_c\,\cdot)$, and
\eqref{eq:resolventperturbation} and $\|q_c-q_0\|_Y\le C\cdot c$ give
$\|\mathcal E_c-\mathcal E_0\|\le C\cdot c$. Its inverse is uniformly bounded by the
Neumann series. The implicit function theorem at every $c>0$ therefore
gives the $C^1$ branch, with overlapping branches identified by local
uniqueness. Differentiation yields
\[
 \mathcal E_c\partial_cq_c=R_c'(q_c^2).
\]
The uniform inverse bound, derivative bound and $H^1$ algebra estimate
give the last inequality in \eqref{eq:branchbounds}.
\end{proof}

\begin{theorem}
\label{thm:smallspeedstability}
There exists $c_*>0$ such that, for every $0<c<c_*$, the centered
minimizing profile is unique, its linearized action satisfies
\begin{equation}\label{eq:smallND}
 \ker L_c=\operatorname{span}\{\phi_c^{\prime}\},
\qquad \phi_c^{\prime}=\partial_x\phi_c.
\end{equation}
and along the resulting branch
\begin{equation}\label{eq:smallslope}
 \frac{d}{dc}\mathcal P(\phi_c)
 =\frac92\sqrt c+O(c^{3/2})>0.
\end{equation}
Its translation orbit is therefore orbitally stable in $H^{3/2}$.
The threshold $c_*$ is one fixed positive number for the normalized
model, chosen from the uniform resolvent, algebra, inverse and slope
bounds in the proof; it does not depend on the individual speed
$c\in(0,c_*)$. No explicit optimal value is claimed.
For each fixed such $c$ and each prescribed orbit tolerance $\varepsilon>0$,
the admissible perturbation radius $\delta=\delta(c,\varepsilon)>0$
may depend on $c$. No radius uniform as $c\downarrow0$ is asserted.
Uniqueness of all positive profiles, or stability at arbitrary positive
speeds, is not asserted.
\end{theorem}
\begin{proof}
The preceding lemma gives uniqueness among centered minimizers and
the branch estimates. On the full space $Y$, the same algebra argument
gives
\[
 \|T_c-T_0\|_{\mathcal L(Y)}\le C\cdot c.
\]
Let $Z_0=\{h\in Y:\langle h,q_0'\rangle=0\}$. If $T_ch=0$ and
$h\in Z_0$, then \eqref{eq:limitinggap} gives
\[
 \|h\|_Y\le C_0\|(T_0-T_c)h\|_Y\le C_0C\cdot c\|h\|_Y.
\]
Shrink $c_*$ so that $C_0C\cdot c_*<1$. Then the preceding
inequality forces $h=0$ for every $0<c<c_*$. A kernel of dimension at least
two would have a nonzero intersection with this codimension-one
space. Thus $\dim\ker T_c\le1$. The smooth profile equation and
translation invariance give $T_cq_c'=0$, with $q_c'\ne0$.
Consequently $\ker T_c=\operatorname{span}\{q_c'\}$.

For $h(x)=k(\sqrt c\,x)$, the rescaled equation $L_ch=0$ is
\[
 [(1-c\partial_y^2)^{1/2}(1-\partial_y^2)-2q_c]k=0,
\]
equivalently $T_ck=0$. An $L^2$ kernel vector of $L_c$ belongs to
its domain $H^3$, hence to $Y$ after scaling. Conversely, for every fixed $c>0$, $R_c:L^2\to H^3$ is bounded,
because its symbol satisfies
$\langle\eta\rangle^3|r_c(\eta)|
=\langle\eta\rangle/(1+c\eta^2)^{1/2}
\le\max\{1,c^{-1/2}\}$.
A $Y$ kernel vector obeys $k=2R_c(q_ck)$; as $q_c\in L^\infty$
and $k\in L^2$, this identity places $k$ in $H^3$, the domain
of the rescaled multiplier. This smoothing bound need not be uniform
as $c\downarrow0$.
Thus the kernels correspond, proving \eqref{eq:smallND}.
The physical even implicit-function construction following
\eqref{eq:ND} now also applies at each of these speeds. We identify
its branch with the rescaled branch explicitly. Fix
$\bar c\in(0,c_*)$, and denote its local $C^1$ even $H^3$ branch by
$\Phi(c)$, with $\Phi(\bar c)=\phi_{\bar c}$. Define
\[
 \widetilde q(c)(y)=c^{-1}\Phi(c)(y/\sqrt c).
\]
Dilations $f\mapsto f(\cdot/a)$ are strongly continuous on $H^1$
for $a>0$ and locally uniformly bounded there: their squared $L^2$
and derivative norms are respectively $a\|f\|_2^2$ and
$a^{-1}\|f'\|_2^2$. Strong continuity follows first for smooth
compactly supported functions and then by density. Since $\Phi$ is
continuous in $H^3$, $\widetilde q$ is continuous in $H^1$ near
$\bar c$. It is even, satisfies
$\widetilde q(c)=R_c(\widetilde q(c)^2)$, and equals $q_{\bar c}$
at $c=\bar c$.
Choose the contraction parameters in Lemma~\ref{lem:smallbranch}
so that $\|q_c-q_0\|_{H^1}<\rho/2$ throughout $(0,c_*)$.
Continuity then keeps $\widetilde q(c)$ in that lemma's fixed
$\rho$-ball around $q_0$ for $c$ near $\bar c$.
Local uniqueness gives $\widetilde q(c)=q_c$, hence
\[
 \Phi(c)(x)=c q_c(\sqrt c\,x).
\]
Thus the scalar function $\mathcal P(\Phi(c))$ is exactly the one
computed from the rescaled branch. Its derivative is therefore the
slope entering Theorem~\ref{thm:conditionalstability}; no differentiability
of the dilation map on arbitrary $H^1$ functions is being assumed.

Set
\[
 M(c)=\int\sqrt{1+c\eta^2}|\widehat q_c(\eta)|^2d\eta,
 \qquad \mathcal P(\phi_c)=\tfrac12c^{3/2}M(c).
\]
The estimate $\|q_c-q_0\|_{H^1}\le C\cdot c$ gives
$\|q_c\|_2^2=6+O(c)$. Also
\[
 0\le M(c)-\|q_c\|_2^2
 \le\tfrac c2\|q_c'\|_2^2,
\]
so $M(c)=6+O(c)$. To control its derivative, differentiate the
quadratic form on $H^1$; the weights are bounded by $1+\eta^2$
for $0<c\le1$, and their derivatives by $\eta^2/2$.
The Fourier transform, viewed as a bounded linear map from $H^1$
to the corresponding weighted $L^2$ space, commutes with the branch
derivative:
\[
 \partial_c\bigl(\mathcal F q_c\bigr)
 =\mathcal F\bigl(\partial_c q_c\bigr).
\]
In hat notation this is $\partial_c\widehat{q_c}
=\widehat{\partial_cq_c}$, with no complex conjugation.
The $C^1$ branch and the preceding bounds justify
\begin{align*}
 M'(c)={}&\frac12\int_{\mathbb R}\frac{\eta^2}{\sqrt{1+c\eta^2}}
                       |\widehat{q_c}(\eta)|^2\,d\eta\\
 &+2\operatorname{Re}\int_{\mathbb R}\sqrt{1+c\eta^2}
       \widehat{q_c}(\eta)\,
       \overline{\partial_c\widehat{q_c}(\eta)}\,d\eta.
\end{align*}
The first integral is bounded by $\|q_c'\|_2^2/2$; weighted
Cauchy--Schwarz and \eqref{eq:branchbounds} bound the second uniformly.
Hence $M'(c)=O(1)$, and
\[
 \frac{d}{dc}\mathcal P(\phi_c)
 =\tfrac34\sqrt c\,M(c)+\tfrac12c^{3/2}M'(c)
 =\tfrac92\sqrt c+O(c^{3/2}).
\]
This is positive for sufficiently small $c>0$.
Proposition~\ref{prop:index} supplies the single negative direction,
and \eqref{eq:smallND} and \eqref{eq:smallslope} verify both hypotheses
of Theorem~\ref{thm:conditionalstability}. The asserted orbital stability
follows from its conserved Lyapunov argument.
\end{proof}
The perturbation verifies the stability assumptions rather than merely
differentiating the earlier asymptotic for $\mathcal P$. Its essential ingredients
are norm convergence of the smoothing resolvent, the KdV limiting kernel,
and the bounded derivative of the rescaled branch. No Sturm comparison
for the nonlocal finite-speed operator has been used.

\section{Conclusions and open questions}
The stability theorem rests on three verified properties of minimizing
profiles: one negative direction, a one-dimensional translation kernel,
and a positive derivative of the quadratic invariant. The variational
construction establishes the negative index at every positive speed.
The KdV limit and smoothing-resolvent perturbation establish the kernel
and slope properties for all sufficiently small positive speeds, and
also give uniqueness of centered minimizers in that regime. The global,
unconditionally unique $H^{1/2}$ flow, its higher-regularity persistence,
and the conserved penalized action then give orbital stability in
$H^{3/2}$.

The fixed-speed construction complements the prescribed-mass nonlocal
variational theory discussed in the introduction. The next questions
are whether the kernel and slope conditions persist at larger speeds,
and whether uniqueness extends to every positive profile. For the
evolution, local existence holds in $L^2$ and unconditional uniqueness
holds for $s\ge1/4$; extending global existence below $H^{1/2}$ requires
an a priori bound beyond the conserved quadratic invariant. Asymptotic
stability remains a separate dynamical question.

\end{document}